\documentclass[11pt]{article}

\usepackage{times}
\usepackage{amsmath,amsfonts,amstext,amssymb,amsbsy,amsopn,amsthm,eucal}
\usepackage{txfonts}
\usepackage{dsfont}
\usepackage{graphicx}   

\newcommand{\Ric}{\text{Ric}}
\newcommand{\Vol}{\text{Vol}}

\newcommand{\NN}{\mathds{N}}
\newcommand{\RR}{\mathds{R}}
\newcommand{\ZZ}{\mathds{Z}}
\newcommand{\Sn}{\mathds{S}}

\newcommand{\cC}{\mathcal{C}}
\newcommand{\cE}{\mathcal{E}}
\newcommand{\cS}{\mathcal{S}}
\newcommand{\cR}{\mathcal{R}}

\begin{document}

\newtheorem{theorem}{Theorem}[section]

\newtheorem{proposition}{Proposition}[section]

\newtheorem{lemma}{Lemma}[section]

\newtheorem{corollary}{Corollary}[section]

\theoremstyle{definition}
\newtheorem{definition}{Definition}[section]

\theoremstyle{remark}
\newtheorem{remark}{Remark}[section]

\theoremstyle{remark}
\newtheorem{example}{Example}[section]

\theoremstyle{remark}
\newtheorem{note}{Note}[section]

\theoremstyle{remark}
\newtheorem{question}{Question}[section]

\theoremstyle{remark}
\newtheorem{conjecture}{Conjecture}[section]

\title{Lower Ricci Curvature, Branching, \\and\\ Bi-Lipschitz Structure of Uniform Reifenberg Spaces}

\author{Tobias Holck Colding and Aaron Naber\thanks{Department of Mathematics, Massachusetts Institute of Technology, Cambridge, MA 02139.  Emails: colding@math.mit.edu and anaber@math.mit.edu.
         The  first author
was partially supported by NSF Grant DMS  0606629, DMS 1104392,  and NSF FRG grant DMS
 0854774 and the second author by an NSF Postdoctoral Fellowship.}}


\date{\today}
\maketitle
\begin{abstract}
We study here limit spaces $(M_\alpha,g_\alpha,p_\alpha)\stackrel{GH}{\rightarrow} (Y,d_Y,p)$, where the $M_\alpha$ have a lower Ricci curvature bound and are volume noncollapsed.  Such limits $Y$ may be quite singular, however it is known that there is a subset of full measure $\cR(Y)\subseteq Y$, called {\it regular} points, along with coverings by the almost regular points $\cap_\epsilon \cup_r\cR_{\epsilon,r}(Y)=\cR(Y)$ such that each of the {\it Reifenberg sets} $\cR_{\epsilon,r}(Y)$ is bi-H\"older homeomorphic to a manifold.  It has been an ongoing question as to the bi-Lipschitz regularity the Reifenberg sets.  Our results have two parts in this paper.  First we show that each of the sets $\cR_{\epsilon,r}(Y)$ are bi-Lipschitz embeddable into Euclidean space.  Conversely, we show the bi-Lipschitz nature of the embedding is sharp.  In fact, we construct a limit space $Y$ which is even uniformly Reifenberg, that is, not only is each tangent cone of $Y$ isometric to $\RR^n$ but convergence to the tangent cones is at a uniform rate in $Y$, such that there exists no $C^{1,\beta}$ embeddings of $Y$ into Euclidean space for any $\beta>0$.  Further, despite the strong tangential regularity of $Y$, there exists a point $y\in Y$ such that every pair of minimizing geodesics beginning at $y$ branches to any order at $y$.  More specifically, given {\it any} two unit speed minimizing geodesics $\gamma_1$, $\gamma_2$ beginning at $y$ and {\it any} $0\leq \theta\leq \pi$, there exists a sequence $t_i\to 0$ such that the angle $\angle \gamma_1(t_i)y\gamma_2(t_i)$ converges to $\theta$.

\end{abstract}

\section{Introduction}

In this paper we are interested in pointed Gromov-Hausdorff limits
\begin{align}\label{con:Y_limit}
(M^n_\alpha,g_\alpha,p_\alpha)\stackrel{GH}{\rightarrow} (Y^n,d_Y,p)
\end{align}
such that the $M_\alpha$'s are $n$-dimensional and satisfy the lower Ricci bound
\begin{align}\label{con:lower_Ricci}
 \Ric(M_\alpha)\geq -(n-1)\, ,
\end{align}
and the noncollapsing assumption
\begin{align}\label{con:noncollapsed}
 \Vol(B_1(p_\alpha))\geq v>0\, .
\end{align}

Given a point $y\in Y$, a tangent cone at $y$ is a metric space $Y_y$, such that for some sequence $r_i\rightarrow 0$ the limit
\begin{align}
 (Y,r_i^{-1}d_Y,y)\stackrel{GH}{\rightarrow}(Y_y,d,y)
\end{align}
exists.  Tangent cones exist at every point and by (\ref{con:noncollapsed}) are metric cones, though they need not be unique, see \cite{ChC2}.  However, it also follows from \cite{ChC2} that there exists a set $\cS\subseteq Y$ satisfying
$$
\dim \cS\leq n-2\, ,
$$
where $\dim$ refers to the Hausdorff dimension, such that for every $y\in \cR(Y)\equiv Y\setminus S$ the tangent cones at $y$ are unique and isometric to $\RR^n$.  We call such points regular points.  The following definition is a uniform version of a regular point, and makes for a convenient notion when studying $\cR$.  The point $0^n\in\RR^n$ will represent the origin in Euclidean space.

\begin{definition}
 For every $\epsilon,r>0$ let us define $\cR_{\epsilon,r}(Y)\subseteq Y$ by
\begin{align}
 \cR_{\epsilon,r}(Y)\equiv\{y\in Y: \forall s<r\text{ we have }d_{GH}(B_s(y),B_s(0^n))<\epsilon s\}\, .
\end{align}
We call such points $(\epsilon,r)$-Reifenberg points.  We say that $Y$ is a uniform Reifenberg space if for every $\epsilon>0$ there is an $r>0$ such that  $Y\equiv \cR_{\epsilon,r}(Y)$.
\end{definition}

A space $Y$ being a uniform Reifenberg space is the statement that not only is every tangent cone of $Y$ isometric to $\RR^n$, but that convergence to these tangent cones is uniform in $Y$.  Being a uniform Reifenberg space is nearly the statement that $Y$ is a Riemannian manifold.  Theorems \ref{t:bilipschitzembedding}, \ref{t:nodiffeomorphism} and \ref{t:nonunique_angle} analyze to what extent this statement is true and false.

A basic property of $\cR_{\epsilon,r}(Y)$ is that
$$
\bigcap_{\epsilon>0}\bigcup_{r>0}\cR_{\epsilon,r}(Y) = \cR(Y)\, .
$$
The regular points represent the points of $Y$ which are {\it infinitesimally} Euclidean.  In fact, it was proved in \cite{ChC2} that for $\epsilon\leq \epsilon(n)$ that $\cR_{\epsilon,r}(Y)$ is $C^{0,\alpha}$ bi-H\"older homeomorphic to a Riemannian manifold for some $0<\alpha<1$.  It was a question as to the bi-Lipschitz structure of such spaces.  The first main theorem of this paper is the following.

\begin{theorem}\label{t:bilipschitzembedding}
 Let $(M^n_\alpha,g_\alpha,p_\alpha)\stackrel{GH}{\rightarrow} (Y^n,d_Y,p)$ with the $M^n_\alpha$ satisfying (\ref{con:lower_Ricci}) and (\ref{con:noncollapsed}).  Then there exists $\bar\epsilon(n),\bar r(n)>0$ such that for every $\epsilon\leq \bar\epsilon$ and $r\leq \bar r$ we have that bounded subsets of $\cR_{\epsilon,r}(Y)$ are bi-Lipschitz embeddable into some Euclidean space.  In particular, if $Y$ is a compact uniform Reifenberg space then $Y$ is bi-Lipschitz embeddable into some Euclidean space.
\end{theorem}

Now that we have seen some extent to which a uniform Reifenberg space $Y$ behaves like a Riemannian manifold, we would now like to see to what extend it doesn't.  To begin with let us see that there exists uniform Reifenberg spaces such that the bi-Lipschitz structure of Theorem \ref{t:bilipschitzembedding} is sharp.  We can interpret this as follows.  We know by \cite{ChC2} that $Y$ is $C^{0,\beta}$-bi-H\"older to a manifold, and by Theorem \ref{t:bilipschitzembedding} that the geometry on $Y$ is induced by a $C^{0,1}$-lipschitz structure.  However, the next theorem tells us that there exist examples of uniform Reifenberg spaces where the geometry need not be induced by a $C^{1,\beta}$-structure for any $\beta>0$.  More precisely we have:

\begin{theorem}\label{t:nodiffeomorphism}
 For $n\geq 3$ there exists $(M^n_\alpha,g_\alpha,p_\alpha)\stackrel{GH}{\rightarrow} (Y^n,d_Y,p)$ with the $M^n_\alpha$ satisfying (\ref{con:lower_Ricci}) and (\ref{con:noncollapsed}), and such that
 \begin{enumerate}
 \item $Y$ is homeomorphic to $\RR^n$.  In fact, for each $0<\beta<1$ we have that $Y$ is $C^{0,\beta}$-bi-H\"older to $\RR^n$.
 \item Every tangent cone of $Y$ is isometric to $\RR^n$.  In fact, $Y$ is a uniform Reifenberg space.
 \item There does not exist a homeomorphism $\phi:\RR^n\to Y$ such that the pullback geometry is induced by a $C^{0,\beta}$-metric for any $0<\beta<1$.
 \end{enumerate}
\end{theorem}

To explore some further properties of the example constructed in Theorem \ref{t:nodiffeomorphism}, let us recall the notion of an angle between minimizing geodesics and discuss some basic properties.  If $X$ is a length space and $x,y,z\in X$, then recall the angle $\angle xyz$ is defined by
\begin{align}
 \cos\angle xyz\equiv\frac{d(x,y)^2+d(y,z)^2-d(x,z)^2}{2d(x,y)d(y,z)}\, .
\end{align}
On the other hand, if $\gamma_1$, $\gamma_2$ are unit speed minimizing geodesics beginning at $x$, then we can attempt to define the angle between the geodesics at $x$ by
\begin{align}
 \angle\dot\gamma_1\dot\gamma_2\equiv \lim_{t\to 0}\angle\gamma_1(t)x\gamma_2(t)\, ,
\end{align}
if the limit exists.  It certainly doesn't need to be the case that the angle between two geodesics is well defined.  It is known that if $X$ is an Alexandrov space, e.g. is a limit of manifolds with a lower sectional curvature bound, then the angle is always well defined; see, for instance \cite{BGP}.  As a consequence of this in an Alexandrov space geodesics do not branch; see page 384 of \cite{GvPe} and \cite{BGP}.
A geodesic is said to be branching if there exists another geodesic that coincide with $\gamma$ on a open
subset, but  that at some point the two curves depart (branch) from each other.  Precisely, there does not
exists a common extension of the two geodesics.  Obviously, for smooth or even $C^{1,\beta}$ manifolds branching cannot occur as geodesics are entirely determined
by their initial conditions (initial velocity).  However, for general limits of manifolds with lower Ricci curvature bounds it is unknown whether or not
geodesics can branch in the interior.

Since tangent cones of even noncollapsing manifolds with lower Ricci curvature bounds are in general nonunique one cannot in general define angles.  However, it has been a question as to whether the angles are well defined at regular points.  More restrictively, we could even ask if angles are well defined if the limit is a uniform Reifenberg space.  The following answers this in the negative and show, in particular, that there are noncollapsed limit spaces where geodesics branch, or equivalently are tangent, at their initial point.

\begin{theorem}\label{t:nonunique_angle}
 The example $Y$ from Theorem \ref{t:nodiffeomorphism} has the following property.  Given any unit speed minimizing geodesics $\gamma_1,\gamma_2$ beginning at $p$ and given any $0\leq\theta\leq\pi$, there exists a sequence $t_a\to 0$ such that the angles $\angle \gamma_1(t_a)x\gamma_2(t_a)$ satisfy
\begin{align}
 \lim_{t_a\to 0}\angle \gamma_1(t_a)x\gamma_2(t_a) =\theta\, .
\end{align}
\end{theorem}

The above is telling us that not only are angles not well defined in $Y$, but they are ill-defined for every pair of geodesics.  In fact, we can find sequences so that the angle between any two geodesics converge to whatever we wish, including $0$.  The above has the consequence that even in the especially nice scenario of a uniform Reifenberg space we cannot expect geodesics to be well defined by their initial conditions.

\section{Proof of Theorem 1.1}\label{s:Proof_BilipEmbed}

In this section we prove Theorem \ref{t:bilipschitzembedding}.  For the rest of this section we fix $\epsilon,r>0$.  For notational simplicity we will actually try and prove that the set $\cR_{\epsilon,100r}(Y)$ is bi-Lipschitz embeddable.  A first goal will be to construct a bi-Lipschitz embedding $\cR_{\epsilon,100r}(Y)\to L^2(Y)$ of the $(\epsilon,100r)$-Reifenberg points into the Hilbert space of $L^2$ functions on $Y$.  Let us begin by defining the embedding of interest.

\begin{definition}
 For $y\in Y$ let us associate the function $\rho_y\in L^2(Y)$ by
\begin{align}
\rho_y(z)\equiv \left\{ \begin{array}{rl}
  d_Y(x,y) &\mbox{ if } d_Y(x,y)\leq 10r \notag\\
  0 &\mbox{ if } d_Y(x,y)>10r
       \end{array} \right. ,\,
\end{align}
\end{definition}

The mapping $\rho:Y\to L^2(Y)$ is clearly a continuous map.  Seeing that $\rho$ is a bi-Lipschitz embedding roughly comes down to showing that there are a lot of points for which a reverse triangle inequality holds.  The following simple lemma will be a handy tool in that direction.  Its role will be to state that whenever we have points  where such a reverse triangle inequality holds, then we can find many more points where the same reverse triangle inequality will continue to hold.

\begin{lemma}\label{l:reverse_triangle}
 Let $x,y,z\in Y$ be such that the following hold:
\begin{enumerate}
 \item $|d(x,z)-d(z,y)|\geq\epsilon d(x,y)$.
 \item $d(x,z)\geq d(z,y)$.
 \item There exists a unit speed minimizing geodesic $\gamma:[0,t]\rightarrow Y$ such that $\gamma(0)=x$ and $\gamma(d(x,z))=z$.
\end{enumerate}
Then for every $s\in [d(x,z),t]$ we have that $|d(x,\gamma(s))-d(\gamma(s),y)|\geq\epsilon d(x,y)$.
\end{lemma}
\begin{proof}
 A triangle inequality gives us
\begin{align}
 d(x,\gamma(s))=d(x,z)+d(z,\gamma(s))\, ,\\
 d(y,\gamma(s))\leq d(y,z)+d(z,\gamma(s))\, .
\end{align}
Combining these yield the lemma.
\end{proof}

Let us now see that the mapping $\rho:Y\to L^2(Y)$ is bi-Lipschitz on $\cR_{\epsilon,100r}(Y)$.  Further, we can control the bi-Lipschitz constant in terms of only the dimension $n$ and $r>0$.

\begin{lemma}
For $\epsilon\leq \bar\epsilon(n)$ and $r\leq \bar r(n)$ we have that $\rho:\cR_{\epsilon,100r}(Y)\to L^2(Y)$ is a bi-Lipschitz map.  In fact, there exists $C(n,r)>1$ such that
\begin{align}
C^{-1}\min\{d(x,y),1\}\leq ||\rho_x-\rho_y||_{L^2(Y)}\leq C d(x,y)\, .
\end{align}
\end{lemma}

\begin{proof}
Let us recall that by the volume convergence, \cite{C}, \cite{ChC2}, if $s<100r$, then there exists $\eta(\epsilon,r:n)>0$, which tends to zero as $\epsilon$ and $r$ do, such that if $x\in\cR_{\epsilon,100r}(Y)$ then
\begin{align}\label{e:vol_pinch}
1-\eta\leq \frac{\Vol(B_s(x))}{\Vol_{-1}((B_s(0^n))}\leq 1\, ,
\end{align}
where $\Vol_{-1}$ is the volume of a ball in hyperbolic space.  Now to begin with we see by the triangle inequality that
\begin{align}
||\rho_x-&\rho_y||^2_{L^2(Y)}=\int_Y |\rho_x(z)-\rho_y(z)|^2\leq \int_{B_{20r}(x)\cup B_{20r}(y)}|d(x,z)-d(z,y)|^2\notag\\
&\leq\left(\Vol(B_{20r}(x))+\Vol(B_{20r}(x))\right)d(x,y)^2\leq C(n,r)d(x,y)^2\, ,
\end{align}
which gives us the upper bound immediately.  Let us also observe that if $r<d(x,y)$ that the lower bound is also immediate, namely:
\begin{align}
\int_Y |\rho_x(z)-&\rho_y(z)|^2\geq \int_{B_{r/4}(y)}|d(x,z)-d(z,y)|^2\notag\\
&\geq \Vol(B_{r/4}(y))\left(\frac{1}{2}r\right)^2\geq C^{-1}(n,r)\, .
\end{align}
What is left is to deal with the case when two points $x,y\in Y$ are close and satisfy $d(x,y)\leq r$.  Let us begin with some terminology.  Given $x,y\in Y$ we let $\gamma_{x,y}:[0,d(x,y)]\rightarrow Y$ denote a unit speed minimizing geodesic between $x$ and $y$.  We may use $d_{x,y}$ as shorthand to denote the distance between $x$ and $y$.  We have two sets of particular interest we wish to define.
\begin{align}
&\cC^{s}_{t}(x)\equiv\{\gamma_{x,z}(s d_{x,z}):z\in B_t(x)\}\, ,\\
&\cE^{s}_{t}(x,U)\equiv\{z\in B_{t}(x):\gamma_{x,z}(s d_{x,z})\in U\text{ for some }\gamma_{x,z}\}\, ,
\end{align}

Note that $\cC^{s}_{t}(x)\subseteq B_s(x)$ and represents the contraction of the ball $B_t(x)$ under the gradient flow by the distance function at $x$.  The set $\cE^s_t(x,U)$ represents the opposite, the expansion the set $U$ by flowing backwards under the gradient flow, at least for those points of $U$ for which this flow exists.  By volume monotonicity and (\ref{e:vol_pinch}) we know that
\begin{align}\label{e:C_volume}
1-\eta(\epsilon,r:n)\leq \frac{\Vol(\cC^s_{10r}(x))}{\Vol(B_{10r}(x))} \leq 1\, .
\end{align}

We wish to exploit this as follows.  Let $x,y\in Y$ with $d(x,y)\leq r$.  Consider the ball $B_{\frac{d_{x,y}}{4}}(y)$, and let us note by (\ref{e:C_volume}) that for $\eta$ sufficiently small, depending only on dimension, that
\begin{align}
\frac{\Vol(\cC^{d_{x,y}/r}_{10r}(x)\cap B_{d_{x,y}/4}(y))}{\Vol(B_{d_{x,y}/4}(y))}>\frac{1}{2}\, .
\end{align}
This is roughly the statement that at least half the points of $B_{d_{x,y}/4}(y)$ are on minimizing geodesics from $x$ which extend out into the ball $B_{10r}(x)$.  The same volume monotonicity and (\ref{e:C_volume}) tells us that
\begin{align}
\frac{\Vol(\cE^{d_{x,y}/r}_{10r}(x,B_{d_{x,y}/4}(y)))}{\Vol(B_{10r}(x))}>\delta(n)\, .
\end{align}

Now let $z\in B_{d_{x,y}/4}(y)$, then a quick computation yields that $|d(x,z)-d(z,y)|>\frac{1}{2}d(x,y)$.  On the other hand we then know by Lemma \ref{l:reverse_triangle}, that if $z\in \cE^{d_{x,y}/r}_{10r}(x,B_{d_{x,y}/4}(y))$ then it still holds that
\begin{align}
|d(x,z)-d(z,y)|>\frac{1}{2}d(x,y)\, .
\end{align}
Combining these observations yields
\begin{align}
||\rho_x-&\rho_y||^2_{L^2(Y)}=\int_Y |\rho_x(z)-\rho_y(z)|^2\geq \int_{\cE^{d_{x,y}/r}_{10r}(x,B_{d_{x,y}/4}(y))}|d(x,z)-d(z,y)|^2\notag\\
&\geq\Vol(\cE^{d_{x,y}/r}_{10r}(x,B_{d_{x,y}/4}(y)))\frac{1}{4}d(x,y)^2\geq C(n,r)d(x,y)^2\, ,
\end{align}
which finishes the proof.
\end{proof}

Now we have a mapping $\rho:Y\rightarrow L^2(Y)$ such that its restriction to $\cR_{\epsilon,100r}(Y)$ is uniformly bi-Lipschitz.  If $U\subseteq \cR_{\epsilon,100r}(Y)$ is any bounded subset of $\cR_{\epsilon,100r}(Y)$, then in particular the restriction $\rho:\bar U\rightarrow L^2(Y)$ to the closure of $U$ is uniformly bi-Lipschitz.  We now rely on the results of \cite{MoWe}, specifically remark 5.II and theorem 5.3, to conclude that there exists a finite dimensional subspace $\RR^N\approx H\subseteq L^2(Y)$ such that the composition $p_H\circ\rho:\bar U\to H$, where $p_H$ is the projection to $H$, remains bi-Lipschitz.

\section{Proof of Theorems 1.2 and 1.3}

In this section we prove Theorems \ref{t:nodiffeomorphism} and \ref{t:nonunique_angle}.  The construction of the example $Y$ is based on an effective version of a principal introduced by the authors in \cite{CN2}.  In \cite{CN2} the primary goal was to take a smooth family of compact Riemannian manifolds $(X,g_s)_{s\in(-\infty,\infty)}$ with
\begin{align}
 \Ric[X_s]\geq n-1\, ,
\end{align}
and such that the volume
\begin{align}\label{con:constant_volume}
 \Vol(X_s)\equiv V\, ,
\end{align}
is independent of $s$, and construct a limit space $Y$ satisfying (\ref{con:Y_limit}),(\ref{con:lower_Ricci}) and (\ref{con:noncollapsed}) such that at some point each of the spaces $C(X_s)$ arose as a tangent cone.  Now we rely on the same basic point, but strengthen (\ref{con:constant_volume}) to the assumption that
\begin{align}\label{con:constant_volumeform}
dv_{g_s}\equiv dv_g\, ,
\end{align}
hence the volume forms are independent of $s$.  The next lemma was proved in \cite{CN2}.

\begin{lemma}\label{l:ExamTech}
Let $X^{n-1}$ be a smooth compact manifold with $g(s)$, $s\in (-\infty,\infty)$, a family of metrics with $h_\infty<1$ such that:
\begin{enumerate}
\item $\Ric[g(s)]\geq (n-2)g(s)$.
\item $\frac{d}{ds}dv(g(s))=0$, where $dv$ is the associated volume form.
\item $|\partial_{s}g(s)|,|\partial_{s}\partial_{s}g(s)|\leq 1$ and $|\nabla\partial_{s}g(s)|\leq 1$, where the norms are taken with respect to $g(s)$.
\end{enumerate}

Then there exist functions $h:\mathds{R}^{+}\rightarrow(0,1)$ and $f:\RR^{+}\rightarrow(-\infty,\infty)$ with $lim_{r\rightarrow 0}h(r)=1$, $lim_{r\rightarrow \infty}h(r)=h_\infty$, $lim_{r\rightarrow 0}f(r)=-\infty$, $lim_{r\rightarrow \infty}f(r)=\infty$ and $lim_{r\rightarrow 0,\infty}rf'(r)=0$ such that the metric $\bar{g}=dr^{2}+r^{2}h^{2}(r)g(f(r))$ on $(0,\infty)\times M$ satisfies $\Ric[\bar{g}]\geq 0$.

Further if for some $T\in (-\infty,\infty)$ we have that $g(s)=g(T)$ for $s\leq T$ then we can pick $h$ such that for $r$ sufficiently small $h(r)\equiv 1$.
\end{lemma}

The goal of the construction is then the following.  We will construct a family of metrics $\{g_s\}$ on the $(n-1)$-sphere $\Sn^{n-1}$.  In fact, each element of $\{(\Sn^{n-1},g_s)\}$ will be isometric to the standard sphere $(\Sn^{n-1},g)$.  However, as tensors it will hold that
$$dv_{g_s}\equiv dv_g\, ,$$
while
$$g_s\neq g\, .$$
That is, each $g_s$ will differ from $g$ by a volume form preserving diffeomorphism.  These diffeomorphisms, though volume form preserving, will be sufficiently degenerate that for any fixed $x,y\in \Sn^{n-1}$ we can find metrics in $\{g_s\}$ to make the distance $d(x,y)$ between $x$ and $y$ as close to any number between $0$ and $\pi$ that we wish.  The metric space $Y$ from Theorems \ref{t:nodiffeomorphism} and \ref{t:nonunique_angle} will end up being the metric constructed from this family and Lemma \ref{l:ExamTech}.  The following proposition constructs our desired family of diffeomorphisms.

\begin{proposition}\label{p:Sn_diffeos}
 There exists a smooth family of diffeomorphisms, $\varphi_s:\Sn^n\rightarrow\Sn^n$ $s\in (-\infty,\infty)$, such that
 \begin{enumerate}
 \item $\varphi_sdv_g = dv_g$, where $dv_g$ is the standard volume form on $\Sn^n$.
 \item For every integer $m\in\ZZ$ we have that $\varphi_{[m-\frac{1}{4},m+\frac{1}{4}]}=id$, the identity map.
 \item For every $x,y\in\Sn^n$, $0\leq\theta\leq\pi$ and $\epsilon>0$ there exists $s \in (-\infty,\infty)$ such that $|\pi-d_g(\varphi_{s}(x),\varphi_{s}(y))| <\epsilon$, where $d_g$ is the standard distance on $\Sn^n$.
 \item If $g_s\equiv\varphi_s^*g$, then $|\partial_{s}g_s|,|\partial_{s}\partial_{s}g_s|\leq 1$ and $|\nabla\partial_{s}g_s|\leq 1$, where the norms are taken with respect to $g_s$.
 \end{enumerate}
\end{proposition}

\begin{proof}
Let's begin with the following smaller claim:  For every $x,y\in\Sn^n$ and every $\epsilon>0$ there exists a smooth family of diffeomorphisms $\phi_s$, $s\in [0,1]$, such that
 \begin{enumerate}
 \item $\phi_sdv_g = dv_g$, where $dv_g$ is the standard volume form on $\Sn^n$.
 \item $\phi|_{[0,\frac{1}{4}]} = \phi|_{[\frac{3}{4},1]}=id$ $\forall s$.
 \item $\phi_s|_{B_\epsilon(x)}=id$ and $B_\epsilon(\phi_s(y))=\phi_s(B_\epsilon(y))$.
 \item For each $0\leq\theta\leq\pi$ there exists $s$ such that $|\theta-d_g(\phi_{s}(x),\phi_{s}(y))| \leq 4\epsilon$.
 \end{enumerate}

To prove the claim we can assume $x$ is the north pole without loss.  Let $\cC$ be a great circle in $\Sn^n$ which passes through $y$ and such that $d_g(x,\cC)=4\epsilon$.  Let $K$ be a rotational killing field on $\Sn^n$ which rotates $\cC$ and who gradient vanishes on $\cC$.  If $\chi:(0,1)\to \RR$ is a cutoff function such that $\chi(t)=1$ for $t\leq \epsilon$ and $\chi(t)=0$ for $t>2\epsilon$, then we can define the bump function $b(z)\equiv \chi(d_g(z,\cC))$ on $\Sn^n$.  Notice this bump function is invariant under the rotations generated by $K$.

Now let us define the vector field $\bar K\equiv b\cdot K$.  Notice that a simple computation gives
\begin{align}
{\rm div}\bar K(z) = b(z){\rm div}K(z)+\chi'\cdot\langle\nabla d_g(z,\cC),\nabla K(z)\rangle = 0\, .
\end{align}
In particular, the $1$-parameter family of diffeomorphisms generated by $\bar K$ preserve the volume form $dv_g$.  We also observe that the third condition of the claim is automatically satisfied by any diffeomorphism generated by $\bar K$.  Additionally, by the assumption that $d_g(x,\cC)=4\epsilon$, we have for each $0\leq\theta\leq\pi$ that there exists a point $z\in\cC$ in the great circle with $|d_g(x,z)-\theta|\leq 4\epsilon$.  Hence, by generating diffeomorphisms based on multiples of $\bar K$ we can easily arrange to construct $\phi_s$ as in the claim.

Now we finish the proof of the Lemma by essentially a covering argument.  For every $N\in\NN$ let $\{B_{2^{-N}}(x^N_i)\}$ be a minimal covering of $\Sn^n$.  For each $i\neq j$ let $\phi^N_{ij}$ be the family of diffeomorphisms as in the claim with $x=x^N_i$, $y=x^N_j$ and $\epsilon=2^{-N}$.  Because these diffeomorphisms all begin and end smoothly at the identity, we may union them up smoothly and let $\varphi_s$, $s\in (-\infty,\infty)$ be such a union of these families over all $N$, $i$ and $j$.  We have that $\varphi_s$ clearly satisfies the first two conditions of the proposition, we need to check the third.  So let $x,y\in\Sn^n$ with $\epsilon>0$.  Let $2^{-N}<\frac{1}{16}\epsilon$ with $x^N_i$ and $x^N_j$ such that $x\in B_{2^{-N}}(x^N_i)$ and $y\in B_{2^{-N}}(x^N_j)$.  At some point the family $\phi^N_{ij}$ will appear in the family $\varphi$, hence by the third and fourth claims for $\phi^N_{ij}$ we have some $s$ such that $d_g(\varphi_s(x),\varphi_s(y))<\epsilon$ as claimed.  To satisfy the fourth claim of the Lemma we need only to reparametrize $\varphi_s$.
\end{proof}

Now using the above proposition, for each $\alpha\in\NN$ let $\varphi^\alpha_s$ be the family of diffeomorphisms defined by
\begin{align}
\varphi^\alpha_s\equiv \left\{ \begin{array}{rl}
  \varphi_s &\mbox{ if } s\geq -\alpha \notag\\
  id &\mbox{ if } s\leq -\alpha
       \end{array} \right. ,\,
\end{align}

For each $\alpha$ let $(\Sn^{n-1},g^\alpha_s)$ be the family of metrics on $\Sn^{n-1}$ with $g^\alpha_s\equiv(\varphi^\alpha_s)^*g$, where $g$ is the standard metric on $\Sn^{n-1}$.  Similarly we let $(\Sn^{n-1},g_s)$ be the family defined by $g_s\equiv(\varphi^\alpha_s)^*g$.

The families $g^\alpha_s$ satisfy the conditions of Lemma \ref{l:ExamTech}, hence, for each $\alpha$, we can let $M_\alpha\approx C(\Sn^{n-1})\approx \RR^n$ with $g_\alpha \equiv \bar g$ the metric from Lemma \ref{l:ExamTech}.  Notice that the sequence $(M_\alpha,g_\alpha,p_\alpha)$ are all smooth manifolds, flat near the cone point $p_\alpha$.  Further it holds that
\begin{align}
(M_\alpha,g_\alpha,p_\alpha)\stackrel{GH}{\rightarrow} (Y,d_Y,p)\, ,
\end{align}
where $Y\approx C(\Sn^{n-1})\approx \RR^n$ is the metric space gotten by applying Lemma \ref{l:ExamTech} to the family $(\Sn^{n-1},g_s)$.  It is clear that $Y$ is smooth away from the cone tip $p$, and that the tangent cones of $Y$ are all $\RR^n$.  Using these two points it is not hard to check that $Y$ is a uniform Reifenberg space.  Thus by \cite{ChC2} we have that $Y$ satisfies Theorem \ref{t:nodiffeomorphism}.1.  Although in the coordinates given by Lemma \ref{l:ExamTech} it is clear $Y$ is not smooth at $p$, it may not be immediately clear that $Y$ is not smooth at $p$ in some other coordinate system.  We will see this is not the case however.  In fact, Theorem \ref{t:nodiffeomorphism} says $Y$ cannot even be induced by a $C^{0,\beta}$ metric.

We begin by proving Theorem \ref{t:nonunique_angle}.

\begin{proof}[Proof of Theorem \ref{t:nonunique_angle}]
 Let us use Lemma \ref{l:ExamTech} to identify
$$Y\setminus\{p\}\approx (0,\infty)\times\Sn^{n-1}\, ,$$
equipped with the metric $\bar g\equiv dr^2+r^2h^2(r) g(f(r))$ in this coordinate system.  Notice that every unit speed minimizing geodesic $\gamma$ from $p$ in $Y$ is thus of the form
$$\gamma(t)\equiv (t,y)\, ,$$
where $y\in \Sn^{n-1}$ is some fixed point.  Now let $\gamma_y(t)$ and $\gamma_z(t)$ be two distinct unit speed minimizing geodesics from $p$, and fix $0\leq\theta\leq\pi$.  By Proposition \ref{p:Sn_diffeos} there exists for each $i\in\NN$ a $r_i$ such that
\begin{align}
 |d_{g_{r_i}}(y,z)-\theta|\leq i^{-1}\, ,
\end{align}
where the distance is being measured on $\Sn^{n-1}$ with respect to the metric $g_{r_i}$.  In particular we have that
\begin{align}
 \angle\gamma_y(r_i)p\gamma_z(r_i)\rightarrow \theta\, ,
\end{align}
as claimed.
\end{proof}

We will now move on to the proof of Theorem \ref{t:nodiffeomorphism}.  The proof of Theorem \ref{t:nodiffeomorphism} will itself depend on Theorem \ref{t:nonunique_angle} and the following Lemma.

\begin{lemma}\label{l:angle_estimate}
Let $g$ be a $C^{0,\beta}$ metric on $B_1(0)\subseteq \RR^n$ with $||g||_{C^{0,\beta}}\leq A$ and let $\gamma_1,\gamma_2:[0,t]\rightarrow B_{1/2}(0)$ be unit speed geodesics beginning at $0$, which are minimizing with respect to $g$.  Then there exists constants $\bar t(n,A)$ and $C(n,A)$ so that if $t\leq\bar t$ then $|\angle\gamma_1(t)0^n\gamma_2(t)-\angle\gamma_1(\frac{t}{2})0^n\gamma_2(\frac{t}{2})|\leq Ct^{\beta}$.
\end{lemma}
\begin{remark}
By definition the $C^{0,\beta}$ norm of a metric tensor here as being the maximum of the $C^{0,\beta}$ norm of $g$ and $g^{-1}$ as matrices on $\RR^n$.
\end{remark}
\begin{proof}
First by the $L^{\infty}$ estimate on $g$ and $g^{-1}$ we can pick $\bar t(A)$ so that $B^g_{2\bar t}(x)\subseteq B^{\RR^n}_1(0)$ for each $x\in B_{1/2}(0)$.  Further we point out that we can assume without loss of generality that $g_{ij}(0)=\delta_{ij}$.  To ensure this we need to change by an affine map, which by the $L^\infty$ bound on $g$ will affect angles and euclidean distances by at most a $C(n,A)$ factor and so there is no harm.

Now, for each $x\in B^g_{2\bar t}(0)$ we have $(1-C(A)|x|^\beta_{\RR^n})\delta_{ij}\leq g_{ij}(x)\leq (1+C(A)|x|^\alpha_{\RR^n})\delta_{ij}$, and so by integrating along curves we get for any $x,y\in B^g_{\bar h}(0)$ that
\begin{align}
(1-C|x,y|_g^\beta)^{1/2}\leq \frac{|x,y|_{\RR^n}}{|x,y|_g} \leq (1+C|x,y|_g^\beta)^{1/2}\, .
\end{align}
From the above we immediately get the estimates
\begin{align}\label{e:midpoint}
 |\gamma_1(\frac{t}{2})-\frac{1}{2}\gamma_1(t)|<Ct^{1+\beta/2}\, ,\\
 |\gamma_2(\frac{t}{2})-\frac{1}{2}\gamma_2(t)|<Ct^{1+\beta/2}\, .
\end{align}
Combining these immediately yields
\begin{align}
&|\angle\gamma_1(t)0^n\gamma_2(t)-\angle\gamma_1(\frac{t}{2})0^n\gamma_2(\frac{t}{2})|\, \\ &\leq |\cos^{-1}\left(\frac{2t^2 -d_g(\gamma_1(t),\gamma_2(t))^2}{2t^2}\right)- \cos^{-1}\left(\frac{2(\frac{t}{2})^2 -d_g(\gamma_1(\frac{t}{2}),\gamma_2(\frac{t}{2}))^2}{2(\frac{t}{2})^2}\right)|\, ,\\
&\leq Ct^{\beta}+|\cos^{-1}\left(1-\frac{|\gamma_1(t)-\gamma_2(t)|_{\RR^n}^2}{2t^2}\right) -\cos^{-1}\left(1-\frac{|\gamma_1(\frac{t}{2})-\gamma_2(\frac{t}{2})|_{\RR^n}^2}{2(\frac{t}{2})^2}\right)|\, ,\\
&\leq Ct^{\beta}\, ,
\end{align}
where the last inequality follows from (\ref{e:midpoint}).
\end{proof}

We can now prove the main theorem:

\begin{proof}
[Proof of Theorem \ref{t:nonunique_angle}]
Assume for some $\beta>0$ that such a homeomorphism $\phi$ did exist, which without loss we can assume maps the cone point to the origin.  Let $\gamma_1$,$\gamma_2$ be aribtrary minimizing geodesics from $p$ in $Y$.  By Lemma \ref{l:angle_estimate} we have that for some $C>0$ and all $t$ sufficiently small that
\begin{align}\label{e:angle_converge}
 |\angle\gamma_1(t)0^n\gamma_2(t)-\angle\gamma_1(\frac{t}{2})0^n\gamma_2(\frac{t}{2})|\leq Ct^{\beta/2}\, .
\end{align}
Now we claim that this is a contradiction.  First let us note that if $t_i\equiv 2^{-i}$, then for all $i<j$ large we get the estimate
\begin{align}
 |\angle\gamma_1(t_i),0,\gamma_2(t_i)-\angle\gamma_1(t_j),0,\gamma_2(t_j)|\leq C\sum_{k=i}^j 2^{-\beta k}\, .
\end{align}
In particular, we see that the angles $\angle\gamma_1(t_i),0,\gamma_2(t_i)$ are converging uniformly.  However, by Theorem \ref{t:nonunique_angle} there exists $s_i\to 0$ and $r_i\to 0$ such that as we dilate $Y$ by $s^{-1}_i$ or $r_i^{-1}$ we have that $\gamma_1$,$\gamma_2$ are converging to geodesic rays out of the origin whose angles satisfy
\begin{align}
 \angle\gamma_1(s_i)0\gamma_2(s_i)\to 0\, ,\\
 \angle\gamma_1(r_i)0\gamma_2(r_i)\to \pi\, .
\end{align}
This contradicts (\ref{e:angle_converge}), which gives that the angles between the limiting rays must be the same.
\end{proof}

\end{document}